\documentclass[11pt,english]{article}

\usepackage[margin=2cm,top=1.6 cm, bottom=1.7 cm]{geometry}

\usepackage{amsthm}
\usepackage{amsmath}
\usepackage{amssymb}
\usepackage{setspace}
\usepackage{mathtools}
\usepackage{graphicx}
\usepackage[hidelinks]{hyperref}
\usepackage{cleveref}
\usepackage{enumitem}
\setlist[itemize]{leftmargin=*} 
\usepackage{framed}
\usepackage{subcaption}

\usepackage{floatrow}

\usepackage{floatrow}
\usepackage[T1]{fontenc}
\usepackage{tikz}
\usepackage{mathdots}
\usepackage{xcolor}
\usepackage{diagbox}
\usepackage{colortbl}
\usepackage[absolute,overlay]{textpos}
\usetikzlibrary{calc}
\usetikzlibrary{decorations.pathreplacing}
\usetikzlibrary{positioning,patterns}
\usetikzlibrary{arrows,shapes,positioning}
\usetikzlibrary{decorations.markings}

\def \smvx {circle[radius = .07][fill = black]}

\tikzstyle{edge}=[very thick]
\definecolor{bostonuniversityred}{rgb}{0.8, 0.0, 0.0}
\definecolor{arsenic}{rgb}{0.23, 0.27, 0.29}
\tikzstyle{diredge}=[postaction={decorate,decoration={markings,
		mark=at position .95 with {\arrow[scale = 1]{stealth};}}}]
\usetikzlibrary{automata,positioning}

\newcommand{\defPt}[3]{
	\def \pt {(#1, #2)}
	\coordinate [at = \pt, name = #3];
}

\newcommand{\defPtm}[2]{

	\coordinate [at = #1, name = #2];
}

\newcommand{\fitellipsis}[2] 
{\draw[fill=white] let \p1=(#1), \p2=(#2), \n1={atan2(\y2-\y1,\x2-\x1)}, \n2={veclen(\y2-\y1,\x2-\x1)}
    in ($ (\p1)!0.5!(\p2) $) ellipse [ x radius=\n2/2+0cm, y radius=0.5cm, rotate=\n1];
}

\usepackage[square,sort,comma,numbers]{natbib}
\newcounter{claimcount}
\theoremstyle{plain}

\newtheorem{theorem}{Theorem}
\Crefname{theorem}{Theorem}{Theorems}

\newtheorem*{thm*}{Theorem}
\newtheorem{thm}{Theorem}
\Crefname{thm}{Theorem}{Theorems}
\numberwithin{thm}{section}

\newtheorem*{lem*}{Lemma}
\newtheorem{lem}[thm]{Lemma}
\Crefname{lem}{Lemma}{Lemmas}

\newtheorem*{claim*}{Claim}
\newtheorem{claim}[claimcount]{Claim}
\crefname{claim}{Claim}{Claims}
\Crefname{claim}{Claim}{Claims}

\Crefname{prop}{Proposition}{Propositions}

\newtheorem{cor}[thm]{Corollary}
\crefname{cor}{Corollary}{Corollaries}

\crefname{conj}{Conjecture}{Conjectures}

\Crefname{qn}{Question}{Questions}

\Crefname{obs}{Observation}{Observations}

\Crefname{ex}{Example}{Examples}

\theoremstyle{definition}

\Crefname{prob}{Problem}{Problems}

\Crefname{defn}{Definition}{Definitions}

\newtheorem*{defn*}{Definition}

\Crefname{figure}{Figure}{Figures}

\theoremstyle{remark}

\renewenvironment{proof}[1][]{\begin{trivlist}
\item[\hspace{\labelsep}{\bf\noindent Proof#1.\/}] }{\qed\end{trivlist}}

\newcommand{\eps}{\varepsilon}

\newcommand{\M}{\mathcal{M}}

\expandafter\def\expandafter\normalsize\expandafter{%
    \normalsize
    \setlength\abovedisplayskip{4pt}
    \setlength\belowdisplayskip{4pt}
    \setlength\abovedisplayshortskip{4pt}
    \setlength\belowdisplayshortskip{4pt}
}

\newcommand{\E}{\mathbb{E}}
\renewcommand{\P}{\mathbb{P}}

\title{Clique minors in graphs with a forbidden subgraph}
\author{%
  Matija Buci\'c\thanks{Department of Mathematics, ETH, Z\"urich, Switzerland. Email: \href{mailto:matija.bucic@math.ethz.ch} {\nolinkurl{matija.bucic@math.ethz.ch}}.}
  \and
    Jacob Fox\thanks{Department of Mathematics, Stanford University, Stanford, CA 94305. Email: {\tt jacobfox@stanford.edu}. Research supported by a Packard Fellowship  and by NSF Award DMS-1855635.}
\and  Benny Sudakov\thanks{Department of Mathematics, ETH, Z\"urich, Switzerland. Email:
\href{mailto:benjamin.sudakov@math.ethz.ch} {\nolinkurl{benjamin.sudakov@math.ethz.ch}}.
Research supported in part by SNSF grant 200021\texttt{\_}196965.}
}

 \date{}

\begin{document}
\maketitle
\begin{abstract}
The classical Hadwiger conjecture dating back to 1940's states that any graph of chromatic number at least $r$ has the clique of order $r$ as a minor. Hadwiger's conjecture is an example of a well-studied class of problems asking how large a clique minor one can guarantee in a graph with certain restrictions. One problem of this type asks what is the largest size of a clique minor in a graph on $n$ vertices of independence number $\alpha(G)$ at most $r$. If true Hadwiger's conjecture would imply the existence of a clique minor of order $n/\alpha(G)$. Results of K\"uhn and Osthus and Krivelevich and Sudakov imply that if one assumes in addition that $G$ is $H$-free for some bipartite graph $H$ then one can find a polynomially larger clique minor. This has recently been extended to triangle-free graphs by Dvo\v{r}\'ak and Yepremyan, answering a question of Norin. We complete the picture and show that the same is true for arbitrary graph $H$, answering a question of Dvo\v{r}\'ak and Yepremyan. In particular, we show that any $K_s$-free graph has a clique minor of order $c_s(n/\alpha(G))^{1+\frac{1}{10(s-2) }}$, for some constant $c_s$ depending only on $s$. The exponent in this result is tight up to a constant factor in front of the $\frac{1}{s-2}$ term.
\end{abstract}

\section{Introduction}
A graph $\Gamma$ is said to be a \textit{minor} of a graph $G$ if for every vertex $v$ of $\Gamma$ we can choose a connected subgraph $G_u$ of $G$, such that subgraphs $G_u$ are vertex disjoint and $G$ contains an edge between $G_v$ and $G_{v'}$ whenever $v$ and $v'$ make an edge in $\Gamma$. The notion of graph minors is one of the most fundamental concepts of modern graph theory and has found many applications in topology, geometry, theoretical computer science and optimisation; for more details, see the excellent surveys \cite{lovasz-survey,norin-survey}. Many of these applications have their roots in the celebrated Robertson-Seymour theory of graph minors, developed over more than two decades and culminating in the proof of Wagner's conjecture \cite{R-S}. One of several equivalent ways of stating this conjecture is that every family of graphs, closed under taking minors can be characterised by a finite family of excluded minors. A forerunner to this result is Kuratowski's theorem \cite{kuratowski}, one of the most classical results of graph theory dating back to 1930. In a reformulation due to Wagner \cite{wagner-kuratowski} it postulates that a graph is planar if and only if neither $K_5$ nor $K_{3,3}$ are its minors.

Another cornerstone of graph theory is the famous $4$-colour theorem dating back to 1852 which was finally settled with the aid of computers in 1976 by \cite{four-color}. It states that every planar graph $G$ has chromatic number\footnote{The chromatic number of a graph $G$, denoted $\chi(G)$, is the minimum number of colours required to colour vertices of $G$ so that there are no adjacent vertices of the same colour.} at most four. In light of Kuratowski's theorem, Wagner \cite{wagner-4-col} has shown that in fact the $4$-colour theorem is equivalent to showing that every graph without $K_5$ as a minor has $\chi(G) \le 4$. In 1943 Hadwiger proposed a natural generalisation, namely that every graph with $\chi(G) \ge r$ has $K_{r}$ as a minor. Hadwiger's conjecture is known for $r \le 5$ (for the case of $r=5$ see \cite{hadwiger-5}) but despite receiving considerable attention over the years it is still widely open for $r \ge 6,$  see \cite{hadwiger-survey} for the current state of affairs. 

In this paper we study the question of how large a clique minor can one guarantee to find in a graph $G$ which belongs to a certain restricted family of graphs. A prime example of this type of problems is Hadwiger's conjecture itself. Another natural example asks what happens if instead of restricting the chromatic number we assume a lower bound on the average degree. Note that $\chi(G) \ge r$ implies that $G$ has a subgraph of minimum degree at least $r-1$. So the restriction in this problem is weaker than in Hadwiger's conjecture and we are interested in how far can this condition take us. This question, first considered by Mader \cite{mader} in 1968, was answered in the 80's independently by Kostochka \cite{kost} and Thomason \cite{thom} who show that a graph of average degree $r$ has a clique minor of order $\Theta(r/\sqrt{\log r})$. This is best possible up to a constant factor as can be seen by considering a random graph with appropriate edge density (whose largest clique minor was analysed by Bollob\'as, Catlin and Erd\H{o}s in \cite{BCE}). 

This unfortunately means that this approach is not strong enough to prove Hadwiger's conjecture for all graphs. For almost four decades, bounding the chromatic number through average degree and using the Kostochka-Thomason theorem gave the best known lower bound on the clique minor given the chromatic number. Very recently, Norin, Postle, and Song \cite{norinpostlesong} got beyond this barrier and in a series of works \cite{Postle1, Postle2} Postle obtained the currently best result, by showing that every graph of chromatic number $r$ has a clique minor of size $\Omega(r/(\log\log r)^6)$. This still falls short of proving Hadwiger's conjecture for all graphs. 

However, if we impose some additional restrictions on the graph it turns out we can do much better. One of the most natural restrictions, frequently studied in combinatorics, is to require our graph $G$ to be $H$-free for some other, small graph $H$. This problem was first considered by K\"uhn and Osthus \cite{K-O} who showed that given $s \le t$ every $K_{s,t}$-free graph with average degree $r$ has a clique minor of order $\Omega \left(r^{1+2/(s-1)}/\log^3 r\right).$ The polylog factor in this result was subsequently improved by Krivelevich and Sudakov \cite{B-M} who obtained in a certain sense the best possible bound. They also obtain tight results, for the case of $C_{2k}$-free graphs. These results show Hadwiger's conjecture holds in a stronger form for any $H$-free graph, provided $H$ is bipartite. On the other hand, if $H$ is not bipartite then taking $G$ to be a random bipartite graph shows that the bound of Kostochka \cite{kost} and Thomason \cite{thom} can not be improved.

A natural next question is whether we can do better if we assume a somewhat stronger condition than a bound on the average degree or the chromatic number. A natural candidate is an upper bound on $\alpha(G)$, the size of a largest independent set in $G$. Indeed, the chromatic number of a graph $G$ is at least $n/\alpha(G)$. An old conjecture, which is implied by Hadwiger's conjecture, (see \cite{hadwiger-survey}) states that if $\alpha(G) \le r$ then $G$ has a clique minor of order $n/r$. Duchet and Meyniel \cite{D-M} showed in 1982 that this conjecture holds within a factor of $2$, which was subsequently improved by \cite{fox,B-G,BLW,Pedersentoft,PlummerStiToft,kawa1,kawa2,woodall,maffraymeyniel}, most notably Fox \cite{fox} gave the first improvement of the multiplicative constant $2$. Building upon the ideas of \cite{fox}, Balogh and Kostochka \cite{B-G} obtain the best known bound to date.

In light of these results, Norin asked whether in this case assuming additionally that $G$ is triangle-free allows for a better bound. This question was answered in the affirmative by Dvo\v{r}\'ak and Yepremyan \cite{D-Y} who show that for $n/r$ large enough, a triangle-free $n$-vertex graph with $\alpha(G) \le r$ has a clique minor of order $(n/r)^{1+\frac{1}{26}}.$ They naturally ask if the same holds if instead of triangle-free graphs we consider $K_s$-free graphs. We show that this is indeed the case.

\begin{theorem}\label{thm:main}
Let $s\ge 3$ be an integer. Every $K_s$-free $n$-vertex graph $G$ with $\alpha(G) \le r$ has a clique minor of order at least  $(n/r)^{1+ \frac1{10(s-2)}},$ provided $n/r$ is large enough.
\end{theorem}

For the case of $s=3$ our result has a simpler proof and gives a better constant in the exponent compared to that in \cite{D-Y}. As an additional illustration we use our strategy to obtain a short proof of a result of K\"uhn and Osthus \cite{K-O} about finding clique minors in $K_{s,t}$-free graphs. The above-mentioned two examples put quite different restrictions on the structure of the underlying graph, nevertheless our approach performs well in both cases. This leads us to believe that our strategy, or minor modifications of it, could provide a useful tool for finding clique minors in graphs under other structural restrictions as well. The strategy is simple enough to be worth describing in the Introduction.

\subsection*{Method}
Given a graph $G$ our strategy for finding minors of large average degree goes as follows:
\begin{enumerate}
    \item We independently colour each vertex of $G$ red with probability $p$ and blue otherwise. 
    \item Each blue vertex chooses independently one of its red neighbours (if one exists) uniformly at random.
\end{enumerate}
This decomposes the graph into stars, either centred at a red vertex with leaves being the blue vertices, which have chosen the central vertex as their red neighbour or being isolated blue vertices which had no red neighbours to choose from. We obtain our random minor $\M(G,p)$ by contracting each star into a single vertex and deleting the isolated blue vertices. 

We note that similar strategies were employed by both K\"uhn and Osthus \cite{K-O}, and Dvo\v{r}\'ak and Yepremyan \cite{D-Y}. Our strategy above streamlines their approaches for finding dense minors. This helps us to develop a new way of analysing the outcome, allowing us to answer the above question of Dvo\v{r}\'ak and Yepremyan \cite{D-Y} as well as to obtain simpler proofs of the results of both \cite{K-O} and \cite{D-Y}.

\textbf{Notation:}  For a graph $G$, we denote by $V(G)$ its vertex set and by $E(G)$ its edge set and by $\delta(G)$, $d(G)$ and $\Delta(G)$ its minimum, average and maximum degree. For $v \in V(G)$ let $d_G(v)$ be the degree of $v$ in $G$ (we often write just $d(v)$ when the underlying graph is clear). If $V(G)$ is red-blue coloured, we denote by $d_r(v)$ the number of red neighbours of $v$. For a subset $S \subseteq V(G)$, let $N(S)$ be the set of vertices adjacent to at least one vertex in $S$. For us an $\ell$-path is a path of length $\ell$, which consists of $\ell$ edges and $\ell+1$ vertices. For a path $v_1v_2\ldots v_n$ we say $v_1$ and $v_n$ are its endvertices and $v_2,\ldots,v_{n-1}$ are its internal vertices. 

\section{Setting up the framework and an example}
We begin the section by stating some well-known tools which we are going to use.

\begin{lem}[Chernoff bound, for a proof see \cite{alon-spencer}]\label{lem:chernoff}
Let $X_1, \ldots, X_d$ be independent random variables, taking value $1$ with probability $p$ and $0$ otherwise and let $X= \sum_{i=1}^d X_i$. Then $\P(X>2pd) \le e^{-pd/3}$ and $\P(X<\frac12pd) \le e^{-pd/8}$.
\end{lem}

\begin{lem}\label{lem:maxcut} Every graph $G$ has a spanning bipartite subgraph in which every vertex has degree at least half as big as it had in $G$.
\end{lem}


\begin{thm}[K\"ov\'ari-S\'os-Tur\'an \cite{K-S-T}]\label{KST}
Let $s\le t$ be positive integers. Every bipartite graph with parts of size $n$ and at least $(t-1)n^{2-1/s}+sn$ edges has $K_{s,t}$ as a subgraph.
\end{thm}

We now introduce some notation and give an overview of how we analyse $\M(G,p).$

If $G$ has $n$ vertices, then $M \sim \M(G,p)$ will almost surely have roughly $np$ vertices. In terms of edges, any edge $xy$ of $G$ which got both its endpoints coloured blue will become an edge in $M$ between the red vertices that $x$ and $y$ picked (assuming $x$ and $y$ each have red neighbours). This means that, provided we pick $p$ carefully, $M$ will have roughly as many edges as $G$. The main issue is that $M$ might not be a simple graph. In other words, between some vertices of $M$ there could be many parallel edges and the main part of the analysis of the above process is to control the number of such parallel edges. 

Given $M \sim \M(G,p)$, we say that a $3$-path $vxyu$ in $G$ is \emph{activated} if both vertices $v$ and $u$ got coloured red, both vertices $x$ and $y$ got coloured blue and $x$ chose $v$ and $y$ chose $u$ as their red neighbours. Note that a path $vxyu$ being activated means there is an edge between $v$ and $u$ in $M$. With this in mind our general strategy for the analysis of $\M(G,p)$ will be to try to find a collection $\mathcal{P}$ of many $3$-paths in our graph, such that not too many of these paths have the same endvertices. The chance that a fixed $3$-path activates will be rather small, but the chance that two such paths simultaneously activate is even smaller. This means that with the right choice of parameters we still expect to see many activated paths from $\mathcal{P}$ but only very few between the same pairs of vertices, which lets us conclude there are many edges in $M$ but few parallel ones. The key part of the approach is to choose $\mathcal{P}$ correctly, which will depend on the assumptions we make on our graph $G$. The technical part of the approach is mostly contained in the following two lemmas. The first one gives a lower bound on the probability that a single $3$-path activates while the second gives an upper bound on the chance that several paths with same endvertices activate simultaneously.

\begin{lem}\label{lem:activation-lb}
Let $G$ be a graph with $\Delta(G) \le d$. The chance that a path $vxyu$ activates in $\M(G,p)$ is at least $\frac{1}{2^7d^2}$ given $\frac4d \le p\le \frac12$.
\end{lem}
\begin{proof}
Let us first consider what happens in the colouring stage of our procedure. We say that a colouring is well-behaved (w.r.t. $vxyu$) if $v,u$ got coloured red, $x,y$ got coloured blue and both $d_r(x),d_r(y) \le 2pd+2$. The probability that the right colours got assigned to $v,x,y,u$ is $p^2(1-p)^2$ and given this the probability that $d_r(x) > 2p(d-2)+2$ is by Chernoff bound (\Cref{lem:chernoff}) at most $e^{-p(d-2)/3}\le e^{-pd/4}$ (note that $d \ge \frac4p \ge 8$), since $x$ has at least $d-2$ neighbours which are not yet coloured or are coloured blue. So the probability that a colouring is well-behaved is at least $$p^2(1-p)^2(1-2e^{-pd/4}) \ge p^2/16.$$

Given that we have obtained a well-behaved colouring in the first stage, the probability that $vxyu$ activates is at least $\frac{1}{(2pd+2)^2}\ge \frac{1}{8p^2d^2}$. So putting things together 
$$\P (vxyu \text{ activates}) \ge \frac{1}{8p^2d^2}\cdot \frac{p^2}{16} \ge \frac{1}{2^7d^2}.$$

\vspace{-0.8cm}
\end{proof}

\begin{lem}\label{lem:activation-ub}
Let $G$ be a graph, $d' \le \delta(G)$ and $\mathcal{P}$ be a non-empty collection of $3$-paths between vertices $v$ and $u$ of $G$. If the set of internal vertices of all paths in $\mathcal{P}$ has size $m$ and $2^7 m\log m<pd'$ then the chance that all paths in $\mathcal{P}$ simultaneously activate in $\M(G,p)$ is at most $2p^2/(d'p/4)^m$.
\end{lem}
\begin{proof}
Let us denote the paths in $\mathcal{P}$ by $vx_iy_iu$. Let $X$ denote the set of $x_i$'s and $Y$ the set of $y_i$'s. The condition of the lemma tells us that $|X \cup Y|=m$. Notice that all paths in $vx_iy_iu$ simultaneously activate only if $v,u$ get coloured red, each vertex in $X\cup Y$ gets coloured blue, all vertices in $X$ choose $v$ as their red neighbour and all vertices in $Y$ choose $u$ as their red neighbour. In particular, unless $X \cap Y = \emptyset$ the probability of simultaneous activation is $0$. So we may assume $X \cap Y = \emptyset$.

We say that a colouring is feasible if $v,u$ get coloured red and all vertices in $X \cup Y$ get coloured blue. We say it is well-behaved if in addition all vertices in $X \cup Y$ have at least $(d'-m)p/2$ red neighbours. The probability that a colouring is feasible is $p^2(1-p)^m$. Given that a colouring is feasible the probability that $d_r(v) < p(d'-m)/2$ for a vertex $v\in X \cup Y$ is by Chernoff bound (\Cref{lem:chernoff}) at most $e^{-p(d'-m)/8}\le e^{-pd'/16},$ since $v$ has at least $d'-m$ neighbours for which we still don't know the colour or we know are red. So the probability that a colouring is not well-behaved given that it is feasible is by a union bound at most $$me^{-pd'/16} \le \frac{m}{(d'p)^m},$$ which follows since $2^7 m\log m<pd'$ implies $pd'/\log (pd') > 16m$, using the fact that $m \ge 2$.

Given a well-behaved colouring the probability that each vertex in $X \cup Y$ chooses the right red neighbour is at most $\left(\frac{4}{d'p}\right)^m$ since $(d'-m)p/2 \ge d'p/4.$ 

The chance that all paths in $\mathcal{P}$ activate given that the colouring is feasible is at most probability that all paths in $\mathcal{P}$ activate given that the colouring is well-behaved plus the probability that the colouring is not well-behaved given it is feasible. By the above bounds, this probability is at most
$$ \frac{4^m}{(d'p)^m}+\frac{m}{(d'p)^m}\le \frac{2^{2m+1}}{(d'p)^m}.$$
Finally, this implies that the chance that $\mathcal{P}$ activates is at most
$\frac{2^{2m+1}}{(d'p)^m}\cdot p^2(1-p)^m \le \frac{2^{2m+1}p^2}{(d'p)^m}.$
\end{proof} \vspace{-0.3cm}

\subsection{Minors in \texorpdfstring{$K_{s,t}$}{K s,t}-free graphs}
In this subsection we illustrate our approach by giving a simpler proof of a result of K\"uhn and Osthus \cite{K-O} on dense minors in $K_{s,t}$-free graphs $G$.

A graph $G$ is said to be \textit{almost regular} if $\Delta(G) \le 2\delta(G)$.

\begin{thm}
For $2\le s \le t$ there is a constant $c=c(s,t)>0$ such that every $K_{s,t}$-free, almost regular graph $G$ with average degree $d$ has a minor of average degree at least $cd^{1+\frac{2}{s-1}}.$ 
\end{thm}

\begin{proof}
We choose $c=c(s,t)=\frac{1}{2^{56}st^2}.$ If $cd^{2/(s-1)} \le 1$ then $G$ itself provides us with the desired minor. So we may assume that $d^{2/(s-1)} \ge c^{-1}=2^{56}st^2$. While we will work with the above explicit value of $c$, for the sake of clarity a reader might assume throughout the argument that $d$ is sufficiently large compared to $s$ and $t$.
We have $d \le \Delta(G) \le 2\delta(G) \le 2d$ which in particular gives $\delta(G) \ge d/2$ and $\Delta(G) \le 2d$. By \Cref{lem:maxcut} there is a spanning bipartite subgraph $G'$ of $G$ which has $d(G') \ge d(G)/2=d/2$ and $\delta(G') \ge \delta(G)/2\ge d/4.$ 

Let us first observe a few easy counts, all but the first of which follow due to the fact $G'$ is $K_{s,t}$-free.
\begin{enumerate}
    \item The number of $3$-paths in $G'$ is at most $2nd^3$, where $n=|G'|$.\\
    This follows since the number of edges in $G'$ is at most $nd/2$ and for any choice of an edge as a middle edge of a $3$-path we have at most $4d^2$ choices for its endvertices. 
    \item The number of cycles of length $6$ is at most $\frac43tnd^{5-1/(s-1)}$.\\
    Note that as each path of length $3$, of which there are at most $2nd^3$, completes into a $6$-cycle in at most $t(2d)^{2-1/(s-1)}$ many ways (and we counted each cycle $6$ times). This follows since given a $3$-path from $v$ to $u$ each such cycle corresponds to an edge between neighbourhoods of $v$ and $u$ both of which have size at most $2d$. The subgraph induced by these neighbourhoods is bipartite and $K_{s-1,t}$-free since any $K_{s-1,t}$ together with $v$ or $u$ would constitute a copy of $K_{s,t}$ in $G$. The claimed bound now follows by K\"ov\'ari-S\'os-Tur\'an theorem (\Cref{KST}).

    \item The number of copies of $K_{2,s}$ is at most $ntd^{s}$.\\
    This time we count, for every vertex $v$ the number of stars $K_{1,s}$ with centre in the second neighbourhood $N_2$ and all leaves in the first neighbourhood $N_1$ of $v$. Let $k=|N_2|$ and $d_1,\ldots, d_k$ denote the number of neighbours of vertices in $N_2$ within $N_1$. The number of our stars is equal to $\sum \binom{d_i}{s} \le t\binom{d(v)}{s} \le t\binom{2d}{s}\le 2td^s$ as otherwise we get a $K_{s,t}$ in $G$. Taking the sum over all vertices and noticing we count each $K_{2,s}$ twice we get the claimed bound.
\end{enumerate}

Let us consider $M \sim \M(G',p)$ with $p=2^{12}\sqrt{t}d^{-\frac1{2(s-1)}}\le2^{12}\sqrt{t}c^{\frac14}\le \frac14$. Let $X$ denote the expected number of activated $3$-paths in $M$. Since $G'$ has at least $nd/4$ edges and each contributes a distinct activated $3$-path, provided its endpoints are blue and have a red neighbour, we deduce $$\E X \ge (1-p)^2(1-2(1-p)^{d/4-1})nd/4\ge (1-p)^2(1-2e^{-pd/8})nd/4 \ge nd/8.$$ 

We say a $6$-cycle in $G'$ activates if it consists of two edge-disjoint activated $3$-paths. Let $Y$ count the  activated $6$-cycles. The chance for two edge-disjoint $3$-paths with endvertices $v,u$ to simultaneously activate is at most $2p^2/(dp/16)^4=\frac{2^{17}}{p^2d^4}$, by \Cref{lem:activation-ub} (with $m=4$ and $d'=d/4$), which applies since $2^{12} \le pd$. Each cycle has three possible pairs of such paths so the chance that a $6$-cycle activates is at most $3\cdot \frac{2^{17}}{p^2d^4}$. In particular, $$\E Y \le 2^{19}tnd^{1-1/(s-1)}/p^2.$$ 

Given a star $K_{1,s}$ that appears between neighbourhoods of vertices $v,u \in G'$, we say it is activated for $v$ and $u$ if all $3$-paths between $v$ and $u$ through an edge of the star activated. Let $Z$ count the number of triples consisting of such a \text{star} and vertices $v,u$ such that the star was activated for $v,u$. Each such triple corresponds to a $K_{2,s}$ with a vertex appended to one of its left vertices. In particular, there are at most $4ntd^{s+1}$ plausible triples each of which activates with probability at most  $2p^{2}/(dp/16)^{s+1}=2^{4s+5}p^{1-s}d^{-s-1}$, by \Cref{lem:activation-ub} (with $m=s+1$ and $d'=d/4$). Here we require $2^9(s+1)\log (s+1) \le pd$ in order to be able to apply the lemma, which holds since $pd \ge 2^{12}\sqrt{d}$ and $d$ is large enough compared to $s$. In particular we get $$\E Z \le 2^{4s+7}ntp^{1-s}.$$

For every activated $6$-cycle we delete a middle edge of an activated $3$-path on the cycle, in total deleting at most $Y$ edges. For every activated star we delete one of its edges, in total at most $Z$ edges get deleted. So we are left with at least $X-Y-Z$ edges in $M$ so we know that the expected number of remaining edges is at least 
\begin{align*}
nd/8-2^{19}tnd^{1-1/(s-1)}/p^2-2^{4s+7}ntp^{1-s}&=nd/8-nd/32-2^{4s+7}ntp^{1-s}\\
&\ge nd/8-nd/32-nd/32 = nd/16,
\end{align*}
where the inequality is equivalent to $d \ge 2^{4s+12}tp^{1-s}$ which after plugging in our choice of $p$ is equivalent to $d \ge t^{3-s}2^{48-16s}$ which holds by our choice of $c$.

\textbf{Claim.} After this process, between any two vertices $v,u\in M$ there are at most $s-1$ parallel edges. 
\begin{proof}
To see why this is true consider the bipartite subgraph with left part consisting of blue neighbours of $v$ which choose $v$ as their red neighbour and the right part similarly consists of blue neighbours of $u$ which picked $u$ as their red neighbour. We let the edge set consist of edges of $G'$ which were a middle edge of an activated $3$-path from $v$ to $u$. There are no two independent edges in this graph as otherwise we would have an activated $6$-cycle in which we did not delete an edge. This means that this graph is a star. If the star had size $s$ or more we would get an activated star for which we have not deleted an edge. So this graph is a star with at most $s-1$ edges.
\end{proof} 

This means that, after deleting all but one parallel edge between each pair of vertices, we are left with a simple graph with the expected number of edges at least $\frac{nd}{16s}$. Since we want to show this minor has large average degree, what remains to be done is to control the number of vertices in $M$.
By the Chernoff bound (\Cref{lem:chernoff}), $M$ has more than $2pn$ vertices with probability at most $e^{-np/3}$ and each such outcome can contribute at most $n^2$ edges to our expectation. So we can find an outcome with at most $2pn$ vertices and at least $\frac{nd}{16s}-n^2e^{-np/3}$ edges. This gives us a minor of average degree at least $$\frac{d}{32sp}-np^{-1}e^{-np/3}\ge d^{1+\frac{1}{2(s-1)}}/(2^{17}s\sqrt{t})-p^{-2} \ge cd^{1+\frac{1}{2(s-1)}}$$ 
where in the first inequality we used $np e^{-np/3}\le 1$ since $np \ge dp \ge 12.$
\end{proof} 

\textbf{Remark:} 
The above theorem holds also without the almost regularity condition as shown by \cite{B-M}. The approach presented above with some additional ideas can be used to obtain an alternative, slightly simpler proof of this result. However, since this is a known result and the purpose of including the above argument was mainly for illustration, we chose not to include the details. 

\section{\texorpdfstring{$K_{s}$}{Ks}-minor-free case}

In this section, we prove our main result, \Cref{thm:main}. We state a slightly stronger version below that is more convenient to work with.

\begin{thm}\label{thm:main-r}
Let $s\ge 3$ be an integer, $\eps < \frac1{10(s-2)}$ and $d$ be large enough. Every $K_s$-free graph $G$ without $K_d$ as a minor has $\alpha(G) \ge \frac n{d^{1-\eps}}$.
\end{thm}

To see why this implies \Cref{thm:main} choose $d=(n/\alpha(G))^{1+\frac{1}{10(s-2)}}$. \Cref{thm:main-r} implies that unless $G$ has a $K_d$ as a minor we must have $d \ge (n/\alpha(G))^{\frac{1}{1-\eps}}$ which is impossible since $\frac{1}{1-\eps}> 1+\frac{1}{10(s-2)}$, and completes the argument. Note in particular that one may take the power of $n/\alpha(G)$ in \Cref{thm:main} to be anything smaller than $\frac{1}{1-\eps}$ so smaller than $\frac{1}{1-\frac{1}{10(s-2)}}=1+\frac{1}{10(s-2)-1}$ and we choose the slightly weaker statement for \Cref{thm:main} for clarity.

In this setting we will need to take a slightly different approach when analysing $\M(G,p)$. The reason is that when working with $K_s$-free graphs, we lack a good way of bounding the number of $3$-paths between an arbitrary pair of vertices, which we previously obtained using the fact our graph was $K_{s,t}$-free. In the present setting we will use the fact that $\alpha(G)$ is bounded to show that independent sets must expand well. Additionally, being $K_s$-free allows us to show that we can cover at least a half of any collection of vertices using few independent sets. With these two facts, we fix a vertex, consider a carefully selected collection of $3$-paths starting at this vertex, and then repeat a similar argument as in the previous section to show that we expect to see many non-parallel edges incident to this vertex in $M$. The following standard Ramsey lemma quantifies the second fact mentioned above. We prove it for completeness. 


\begin{lem}\label{lem:ramsey-asym}
Let $s \ge 2$. It is possible to cover half of the vertices of any $n$-vertex $K_s$-free graph with at most $4n^{1-\frac{1}{(s-1)}}$ independent sets.
\end{lem}
\begin{proof}
We will show that every $K_s$-free graph with $m$ vertices contains an independent set of size at least $m^{\frac1{s-1}}/2$. It follows that, as long as we have at least $n/2$ vertices left, we can find an independent set of size at least $n^{1/(s-1)}/4$ and remove it from the graph. Once we stop we removed at most $4n^{1-\frac{1}{(s-1)}}$ independent sets which cover at least half of the graph, as desired. 

We prove the above claim by induction on $s$. For $s=2$ the graph being $K_2$-free means there are no edges so there is an independent set of size $m$ as desired. Assume now that $s \ge 3$ and that the lemma holds for $s-1.$
If there is a vertex which has degree at least $m^{\frac{s-2}{s-1}}$ then, since we know its neighbourhood is $K_{s-1}$-free, we are done by induction. On the other hand, if all vertices have degree less than $m^{\frac{s-2}{s-1}}$, then by Tur\'an's theorem (see \cite{alon-spencer}) we know that there is an independent set of size $\frac m{m^{\frac{s-2}{s-1}}+1}\ge m^{\frac1{s-1}}/2$.
\end{proof} \vspace{-0.3cm}

We will also make use of the result of Kostochka \cite{kost} and Thomason \cite{thom} mentioned in the introduction, which lets one pass from dense minors to clique minors.

\begin{thm}\label{thm:dense-to-clique}
Every graph with average degree at least $(3+o(1))t \sqrt{\log t}$ has $K_t$ as a minor.
\end{thm}

We say that a graph $G$ is \emph{$d$-independent set expanding} if for any independent set $S$ in $G$ we have $|N(S)| \ge (d^{1-\eps}-1) |S|.$ We now prove our main result with a restriction on the maximum degree and assuming that independent sets of our graph expand. Both these assumptions will be easy to remove later.

\begin{thm}\label{thm:main-max-deg}
Let $s\ge 3$ be an integer, $\eps < \frac1{10(s-2)}$ and $d$ be large enough. Every $K_s$-free, $d$-independent set expanding graph $G$ with $\Delta(G)\le d$ has $K_d$ as a minor.
\end{thm}
\begin{proof}
Note first that since independent sets expand we have $\delta(G)\ge d'=d^{1-\eps}-1.$

Fix a vertex $v$. Since $G$ is $K_s$-free we know that $N(v)$ must be $K_{s-1}$-free. By \Cref{lem:ramsey-asym} this means there exist disjoint independent sets of vertices $S_1, \ldots, S_t \subseteq N(v)$ such that $|S_1|+\ldots +|S_t| \ge |N(v)|/2$ and $t \le 4|N(v)|^{1-1/(s-2)}$.

Let $N_i=N(S_i) \setminus (\{v\} \cup N(v))$. For every vertex in $N_i$ we delete all but $1$ edge towards $S_i$, leaving us with at least $|N_i|$
edges between $S_i$ and $N_i$ which split into stars with centres in $S_i$. In particular, this means that the remaining neighbourhoods of vertices in $S_i$ partition $N_i$. For each vertex $w \in S_i$ we apply \Cref{lem:ramsey-asym} to find a collection of at most $t_w \le 4d(w)^{1-1/(s-2)}$ disjoint independent sets $S_{w,1},\ldots,S_{w,t_w}\subseteq N(w)\cap N_i$ which cover at least half of $N(w)\cap N_i$, where we again used the fact that $N(w)$ is $K_{s-1}$-free. For every vertex in $N_{w,j}:=N(S_{w,j})\setminus (\{v\} \cup N(v))$ we mark one of its edges towards $S_{w,j}$ as \textit{permissible} for $S_{w,j}$. See \Cref{fig:cleaning} for an illustration. Note that despite what the figure might indicate we do allow $N_{w,j}$'s (so ``third level'' vertices) to intersect with some vertices among $S_{w,j}$'s (so ``second level'' vertices) however we require both sets to be disjoint from $v$ and $N(v)$.

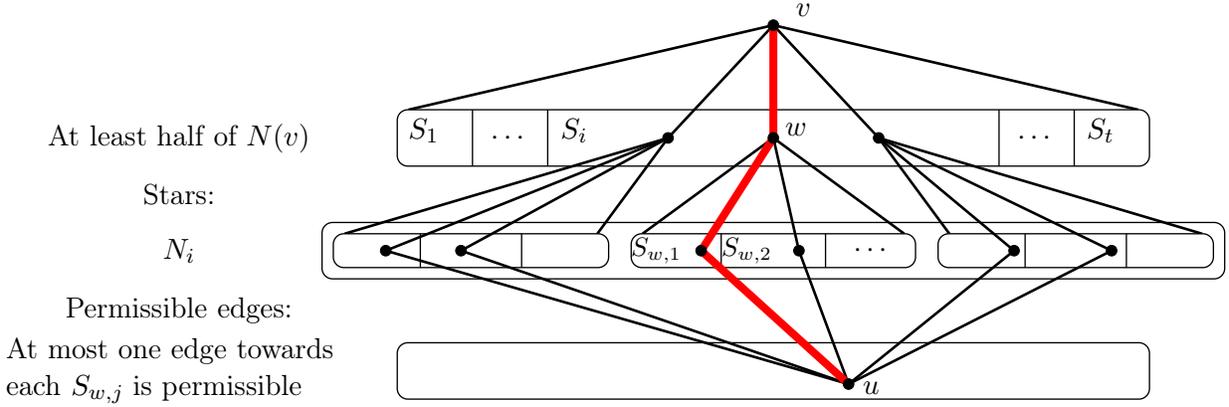
\begin{figure}
\centering
\begin{tikzpicture}[scale = 1,yscale=1, xscale=1]

\defPt{0}{1.25}{dy};
\defPt {-1}{-0.25}{a0};
\defPt {9}{-0.25}{a1};
\defPt {-1}{0.5}{a2};
\defPt {9}{0.5}{a3};
\defPt {-2}{-1.75}{b0};
\defPt {10}{-1.75}{b1};
\defPt {-2}{-1}{b2};
\defPt {10}{-1}{b3};

\defPt {-1}{-3.35}{c0};
\defPt {9}{-3.35}{c1};
\defPt {-1}{-2.6}{c2};
\defPt {9}{-2.6}{c3};

\defPtm{($0.5*(a0)+0.5*(a1)+1.5*(dy)$)}{v};

\draw[line width =1 pt] (v) -- ($(a2)+(0.15,0)$);
\draw[line width =1 pt] (v) -- ($(a3)-(0.15,0)$);

\node at ($0.5*(a0)+0.5*(a2)-(2.9,0)$) { At least half of $N(v)$ };
\node at ($0.5*(b0)+0.5*(b2)-(1.9,0)$) { $N_i$ };

\node at ($0.5*(b2)+0.5*(a0)-(2.4,0)$) { Stars: };
\node at ($0.5*(c2)+0.5*(b0)-(2.4,0)$) {Permissible edges:};
\node[text width = 5cm] at ($0.5*(c0)+0.5*(c2)-(2.7,0)$) { At most one edge towards each $S_{w,j}$ is permissible };

\draw[rounded corners, fill=white, line width =0.5 pt] (a0) rectangle (a3);
\draw[rounded corners, fill=white, line width =0.5 pt] (b0) rectangle (b3);
\draw[rounded corners, fill=white, line width =0.5 pt] (c0) rectangle (c3);

\defPtm{($0.1*(a2)+0.9*(a3)$)}{a13};
\defPtm{($0.1*(a0)+0.9*(a1)$)}{a03};
\defPtm{($0.2*(a2)+0.8*(a3)$)}{a12};
\defPtm{($0.2*(a0)+0.8*(a1)$)}{a02};
\defPtm{($0.8*(a2)+0.2*(a3)$)}{a11};
\defPtm{($0.8*(a0)+0.2*(a1)$)}{a01};
\defPtm{($0.9*(a2)+0.1*(a3)$)}{a10};
\defPtm{($0.9*(a0)+0.1*(a1)$)}{a00};

\defPtm{($0.33*(b2)+0.67*(b3)$)}{b13};
\defPtm{($0.33*(b0)+0.67*(b1)$)}{b03};

\defPtm{($0.67*(b2)+0.33*(b3)$)}{b11};
\defPtm{($0.67*(b0)+0.33*(b1)$)}{b01};

\draw[line width =0.5 pt] (a13) -- (a03);
\draw[line width =0.5 pt] (a12) -- (a02);
\draw[line width =0.5 pt] (a11) -- (a01);
\draw[line width =0.5 pt] (a10) -- (a00);

\node at ($0.5*(a02)+0.5*(a13)$) {$\cdots$};
\node at ($0.5*(a00)+0.5*(a11)$) {$\cdots$};

\node at ($0.5*(a0)+0.5*(a2)+(0.35,0.1)$) {$S_1$};
\node at ($0.5*(a11)+0.5*(a01)+(0.35,0.1)$) {$S_i$};
\node at ($0.5*(a13)+0.5*(a03)+(0.35,0.1)$) {$S_t$};

\defPt{1.6}{0}{ws}
\defPt{1.4}{0}{wd}

\defPtm{($0.5*(a11)+0.5*(a01)+(ws)$)}{w1};
\defPtm{($0.5*(a11)+0.5*(a01)+(ws)+(wd)$)}{w2};
\defPtm{($0.5*(a11)+0.5*(a01)+(ws)+2*(wd)$)}{w3};

\draw[line width =1 pt] (w1) -- ($(b11)-(0.3,0.15)$);
\draw[line width =1 pt] (w1) -- ($(b2)-(-0.3,0.15)$);

\draw[line width =1 pt] (w2) -- ($(b13)-(0.3,0.15)$);
\draw[line width =1 pt] (w2) -- ($(b11)-(-0.3,0.15)$);

\draw[line width =1 pt] (w3) -- ($(b3)-(0.3,0.15)$);
\draw[line width =1 pt] (w3) -- ($(b13)-(-0.3,0.15)$);

\draw[rounded corners, fill=white, line width =0.5 pt] ($(b0)+(0.15,0.15)$) rectangle ($(b11)-(0.15,0.15)$);
\draw[rounded corners, fill=white, line width =0.5 pt] ($(b01)+(0.15,0.15)$) rectangle ($(b13)-(0.15,0.15)$);
\draw[rounded corners, fill=white, line width =0.5 pt] ($(b03)+(0.15,0.15)$) rectangle ($(b3)-(0.15,0.15)$);

\draw[line width =0.5 pt] ($0.67*(b2)+0.33*(b11)-(0,0.15)$) -- ($0.67*(b0)+0.33*(b01)+(0,0.15)$);
\draw[line width =0.5 pt] ($0.33*(b2)+0.67*(b11)-(0,0.15)$) -- ($0.33*(b0)+0.67*(b01)+(0,0.15)$);
\draw[line width =0.5 pt] ($0.67*(b11)+0.33*(b13)-(0,0.15)$) -- ($0.67*(b01)+0.33*(b03)+(0,0.15)$);
\draw[line width =0.5 pt] ($0.33*(b11)+0.67*(b13)-(0,0.15)$) -- ($0.33*(b01)+0.67*(b03)+(0,0.15)$);
\draw[line width =0.5 pt] ($0.67*(b13)+0.33*(b3)-(0,0.15)$) -- ($0.67*(b03)+0.33*(b1)+(0,0.15)$);
\draw[line width =0.5 pt] ($0.33*(b13)+0.67*(b3)-(0,0.15)$) -- ($0.33*(b03)+0.67*(b1)+(0,0.15)$);

\node at ($0.335*(b03)+0.335*(b13)+0.165*(b3)+0.165*(b1)+(-3.7,0)$) {\small $S_{w,2}$};
\node at ($0.5*(b03)+0.5*(b13)-(3.6,0)$) {\small $S_{w,1}$};
\node at ($0.165*(b03)+0.335*(b1)+0.5*(b3)-(4.025,0)$) {$\cdots$};

\defPt{1}{0}{dlta}

\defPtm{($0.5*(b03)+0.5*(b13)+1.16*(dlta)$)}{sw1};

\defPtm{($0.335*(b03)+0.335*(b13)+0.165*(b3)+0.165*(b1)+1.15*(dlta)$)}{sw2};
\draw (sw2) \smvx;

\draw[line width =1 pt] (w3) -- (sw1);
\draw[line width =1 pt] (w3) -- (sw2);

\defPtm{($0.5*(b03)+0.5*(b13)-3*(dlta)$)}{sw4};

\defPtm{($0.5*(b03)+0.5*(b13)-1.7*(dlta)$)}{sw3};
\draw (sw3) \smvx;

\defPtm{($0.335*(b03)+0.335*(b13)+0.165*(b3)+0.165*(b1)-8.5*(dlta)$)}{sw6};
\draw (sw6) \smvx;

\defPtm{($0.335*(b03)+0.335*(b13)+0.165*(b3)+0.165*(b1)-7.5*(dlta)$)}{sw5};
\draw (sw5) \smvx;

\defPt{0}{0.25}{dltay}
\defPtm{($0.4*(c0)+0.6*(c3)-(dltay)$)}{u};

\draw[line width =1 pt] (w2) -- (sw3);
\draw[line width =1 pt] (w2) -- (sw4);

\draw[line width =1 pt] (w1) -- (sw5);
\draw[line width =1 pt] (w1) -- (sw6);

\foreach \i in {1,2,3}
{
\draw[line width =1 pt] (v) -- (w\i);
}

\foreach \i in {1,...,6}
{
    \draw[line width =1 pt] (u) -- (sw\i);
}

\draw[line width = 3 pt,red] (u) -- (sw4) -- (w2) -- (v);

\draw (sw4) \smvx;

\draw (sw1) \smvx;

\draw (u) \smvx;

\node at ($(w2)+(0.3,0.12)$) {$w$};

\foreach \i in {1,2,3}
{
  \draw (w\i) \smvx;   
}

\draw (v) \smvx;   
\node at ($(v)+(0.4,0.2)$) {$v$};
\node at ($(u)+(0.3,-0.05)$) {$u$};
\end{tikzpicture}

    \caption{Illustration of the result of the cleaning procedure for $S_i$, with a $3$-path in $\mathcal{P}_i$ drawn red. Every vertex in $N_i$ has only one edge remaining towards $S_i$, so the remaining edges between $S_i$ and $N_i$ span stars. We mark at most one edge from each vertex $u$ adjacent to $S_{w,j}$ as permissible for $S_{w,j}$.}\label{fig:cleaning}
\end{figure}

We now build a collection of $3$-paths $\mathcal{P}_i$ as follows. Every path in $\mathcal{P}_i$ starts with $v$ then proceeds to a $w \in S_i$ then to a vertex in $S_{w,j}$ for some $j$ along one of the remaining edges between $S_i$ and $N_i$ and finally follows an $S_{w,j}$-permissible edge. Let us first show some properties of $\mathcal{P}_i$ that we will use. 

\begin{claim}\label{claim:stars}
Middle edges of paths in $\mathcal{P}_i$ span stars.
\end{claim} \vspace{-0.5 cm}
\begin{proof}
This is immediate since we removed all but one edge of each vertex in $N_i$ towards $S_i$.
\end{proof} \vspace{-0.5cm}

\begin{claim}\label{claim:multiple-paths}
For any $w \in S_i$ and $u \in V(G)\setminus (\{v\} \cup N(v))$ there are at most $4d^{1-\frac1{s-2}}$ paths in $\mathcal{P}_i$ passing through $w$ and ending with $u$.
\end{claim} \vspace{-0.5 cm}
\begin{proof}
This claim follows since the last edge of any such path in $\mathcal{P}_i$ must be $S_{w,j}$-permissible for some $j$ and any vertex $u$ sends at most one permissible edge towards each $S_{w,j}$. Therefore, there can be at most $t_w \le 4d^{1-\frac1{s-2}}$ such paths in $\mathcal{P}_i$.
\end{proof} \vspace{-0.5cm}

\begin{claim}\label{claim:size}
$|\mathcal{P}_i| \ge d'^2|S_i|/4-2d'd.$
\end{claim} \vspace{-0.5 cm}
\begin{proof}
This follows since any $S_{w,j}$-permissible edge gives rise to a distinct $3$-path in our construction. Since we mark precisely one such edge for each vertex in $N_{w,j}$ we get 

\begin{align*}
    |\mathcal{P}_i|= \sum_{w \in S_i, j \in [t_w]}|N_{w,j}| &\ge \sum_{w \in S_i, j \in [t_w]} (d'|S_{w,j}|-(d+1))\\
    &\ge d'|N_i|/2-|S_i|\cdot 4d^{1-1/(s-2)}(d+1) \\
    &\ge d'^2|S_i|/2-d'(d+1)-|S_i|\cdot 8 d^{2-1/(s-2)} \ge d'^2|S_i|/4-2d'd
\end{align*}
where in the first inequality we used the fact $S_{w,j}$ is independent so by the expansion property $|N_{w,j}| \ge d'|S_{w,j}|-|N(v)|-1$. In the second inequality we used the fact that $\sum_{j \in [t_w]} |S_{w,j}| \ge |N(w) \cap N_i|/2$ and $\cup_{w \in S_i} N(w) \supseteq N_i$ to bound the first term and $t_w \le 4d^{1-1/(s-2)}$ to bound the second. In the third inequality we used $|N_i|\ge |S_i|d'-(|N(v)|+1)$ which holds by the expansion property since $S_i$ is independent. In the last inequality we used $ \eps < \frac{1}{2(s-2)}$ and $d$ (so also $d'$) large enough.  
\end{proof} \vspace{-0.3cm}

Let $\mathcal{P}:=\bigcup_i \mathcal{P}_i.$ Applying the above claim for each $i$ we obtain that for large enough $d$:
\begin{equation}\label{eq:number-of-paths}
|\mathcal{P}| \ge \frac{d'^2}{4}\sum_{i=1}^t|S_{i}|-2td'd\ge \frac{d'^3}{8}-8d'd^{2-1/(s-2)} \ge \frac{d'^3}{16}.
\end{equation}

Let us fix another vertex $u$ and look at the collection $\mathcal{P}_{vu}$ of paths in $\mathcal{P}$ ending in $u$. Let $X$ be the set of vertices following $v$ (so in $N(v)$) on these paths and $Y$ the sets of vertices preceding $u$ on these paths. Since we excluded $N(v)$ from our $N_i$'s we have $X \cap Y = \emptyset.$ Consider the bipartite subgraph $B$ with bipartition given by $X$ and $Y$ and edges coming from middle edges of paths in $\mathcal{P}_{vu}$. 
\begin{claim}\label{claim:max-deg}
$\Delta(B)\le 4d^{1-1/(s-2)}.$
\end{claim} \vspace{-0.3cm}
\begin{proof}
For any $w$ in $X$, we know by \Cref{claim:multiple-paths} that there are at most $4d^{1-1/(s-2)}$ paths in $\mathcal{P}_{uv}$ so $w$ can have at most this many neighbours in $B$. For any vertex $y \in Y$, \Cref{claim:stars} implies $y$ has at most one edge towards any $S_i$; in particular, it has degree at most $t\le 4d^{1-1/(s-2)}$ in $B$.
\end{proof} \vspace{-0.3cm}

Let $p=2^{10}d^{2\eps-\frac{1}{2(s-2)}}$ and $M\sim \M(G,p).$
\begin{claim}\label{claim:bound}
The probability that some path in $\mathcal{P}_{vu}$ activates for $M$ is at least $\frac{|\mathcal{P}_{vu}|}{2^{9}d^2}.$
\end{claim} 
\begin{proof}
Let $a:=|\mathcal{P}_{vu}|$, so by \Cref{claim:max-deg} and the fact that parts of $B$ have size at most $d$, we get $|E(B)|=a \le 4d^{2-1/(s-2)}$. We trivially have at most $a^2$ pairs of independent edges in $B$. On the other hand, the number of pairs of edges sharing a vertex is at most $a\Delta(B)\le 4ad^{1-1/(s-2)}$, since the first edge we can choose in $a$ ways at which point we have $2\Delta(B)$ choices for the second and we count every pair twice.  

Let us denote by $A_e$ the event that edge $e$ in $B$ activated path $veu.$ By the inclusion-exclusion principle,
\begin{align*}
\P\left(\bigcup_{e\in B} A_e\right)
    &\ge \sum_{e\in B}\P(A_e)-\sum_{\{e, f\}\subseteq B}\P(A_e \cap A_f)\\
    & = \sum_{e\in B}\P(A_e)-\sum_{\substack{\{e, f\}\subseteq B, \\e \cap f = \emptyset}}\P(A_e \cap A_f)-\sum_{\substack{\{e, f\}\subseteq B, \\|e \cap f|=1}}\P(A_e \cap A_f)\\
    &\ge a \cdot \frac{1}{2^7d^2}-a^2\cdot \frac{2^9}{p^2d'^4}-4ad^{1-1/(s-2)}\cdot  \frac{2^{7}}{pd'^3} \ge \frac{a}{2^{9}d^2},
\end{align*}
where in the second inequality to bound $\P(A_e)$ we used \Cref{lem:activation-lb} while to bound $\P(A_e \cap A_f)$ we used \Cref{lem:activation-ub} with $m=4$ if $e \cap f=\emptyset$ and with $m=3$ if $|e \cap f|=1$; the conditions of the lemmas are easily seen to hold for our choice of $p$ when $d$ is large enough. In the last inequality we used $\frac{a}{p^2d'^4} \le \frac{4d^{2-1/(s-2)}}{p^2d'^4} \le 8d^{-2+4\eps-1/(s-2)}/p^2= 2^{-17}/d^2$ and $\frac{d^{1-1/(s-2)}}{pd'^3}\le 2^{-11}d^{-2+\eps-\frac1{2(s-2)}} \le 2^{-19}/d^2$, for large enough $d$.  
\end{proof} \vspace{-0.3cm}

Using \Cref{claim:bound} and the lower bound on the total number of paths in $\mathcal{P}$ given in \eqref{eq:number-of-paths} we obtain that the expected number of distinct neighbours of $v$ in $M$ is at least $\frac{\sum_u |\mathcal{P}
_{vu}|}{2^{9}d^2}\ge \frac{d'^3}{2^{13}d^2}\ge 2^{-14}d^{1-3 \eps}.$ As $v$ was a fixed and arbitrary vertex of $G$ in the above argument, we conclude that the expected number of non-parallel edges in $M$ is at least $n2^{-15}d^{1-3 \eps}$. 

By the Chernoff bound (\Cref{lem:chernoff}), $M$ has more than $2pn$ vertices with probability at most $e^{-np/3}$, and each such outcome can contribute at most $n^2$ edges to our expectation. So we can find an outcome with at most $2pn$ vertices and at least $n2^{-15}d^{1-3 \eps}-n^2e^{-np/3}$ edges, so of average degree at least $n2^{-14}d^{1-3 \eps}/(2pn)-np^{-1}e^{-pn/3}\ge 2^{-26}d^{1-5 \eps+\frac{1}{2(s-2)}} \gg d \sqrt{\log d}$ since $\eps<\frac1{10(s-2)}$. So we can find a $K_d$ minor by \Cref{thm:dense-to-clique}.
\end{proof} \vspace{-0.3cm}

Let us first remove the requirement that independent sets expand. 
\begin{cor}\label{cor:main-max-deg}
Let $s\ge 3$ be an integer, $\eps < \frac1{10(s-2)}$ and $d$ be large enough. Every $K_s$-free graph $G$ on $n$ vertices with $\Delta(G)\le d$ and no $K_d$ minor has $\alpha(G) \ge \frac n{d^{1-\eps}}$.
\end{cor}
\begin{proof}
The proof is by induction on the number $n$ of vertices of $G$. The claim is trivial in the base case $n \le d^{1-\eps}$. 
As $G$ has no $K_d$ minor, \Cref{thm:main-max-deg} implies it cannot be $d^{1-\eps}$-independent set expanding. In other words, there is an independent set $S$ in $G$ with $|N(S)| < (d^{1-\eps}-1)|S|$. Then $G'=G \setminus (S \cup N(S))$ still has $\Delta(G') \le \Delta(G) \le d$, is $K_s$-free and has no $K_d$ minor, so by the inductive assumption it has an independent set $S'$ of size $|G'|/d^{1-\eps}$. There are no edges between $S$ and $S'$ since $N(S) \cap G' = \emptyset$, so $S \cup S'$ is an independent set of size $|S|+|G'|/d^{1-\eps} \ge (|S|+N(S)+|G'|)/d^{1-\eps}=n/d^{1-\eps}$, as desired.
\end{proof}

Let us finally remove the restriction on the maximum degree to obtain our main result.

\begin{proof}[ of \Cref{thm:main-r}]
Note that $G$ has average degree at most $d\sqrt{\log d}$ as otherwise \Cref{thm:dense-to-clique} implies it has a $K_d$ minor.
This implies that there can be at most $n/2$ vertices of degree at least $2d \sqrt{\log d}$ in $G$. In particular, there is an induced subgraph of order at least $n/2$ with maximum degree at most $2d \sqrt{\log d}.$ Applying our \Cref{cor:main-max-deg} (with $2d \sqrt{\log d}$ in place of $d$ and for some $\eps'$ which satisfies $\eps< \eps' < \frac1{10(s-2)}$), we obtain $ \alpha(G) \ge \frac{n/2}{(2d\sqrt{\log d})^{1-\eps'}} \ge \frac{n}{d^{1-\eps}},$ using $d$ is large enough and $\eps<\eps'$, completing the proof.
\end{proof} 

\section{Concluding remarks}
We proved that every $K_s$-free graph $G$ on $n$ vertices has a clique minor of order polynomially larger than $n/\alpha(G)$, which would be implied by Hadwiger's conjecture. In particular, we can take any power smaller than $1+\frac{1}{10(s-2)-1}.$ Examples based on random graphs used to bound the Ramsey number $R(s,k)$ for $s$ fixed and $k$ large (see \cite{bohman-keevash} and references therein) are $K_s$-free and have no clique minor of order $(n/\alpha(G))^{1+\frac{1}{s-1}+o(1)}$, showing that our result is best possible up to a constant factor in front of the $\frac{1}{s-2}$ term. In terms of the constant factor, being more careful with our bounds one can easily improve the $1/10$ factor to $1/8$ and with more work even a bit further. It would be interesting to determine the best possible exponent of $n/\alpha(G)$ in our result.

\section*{Acknowledgements}

We are extremely grateful to the anonymous referees for their careful reading of the paper and many useful suggestions and comments.

\vspace{-0.09cm}





\begin{thebibliography}{10}
\bibitem{alon-spencer}
N. Alon and J. H. Spencer, \emph{The probabilistic method.} John Wiley \&
  Sons, 2016.

\bibitem{four-color}
{K. Appel and W. Haken,}
\newblock {\em Every planar map is four colorable.}
\newblock {Bull. Amer. Math. Soc.} 82 (1976), 711--712.

\bibitem{B-G}
{J. Balogh and A. Kostochka,}
\newblock {\em Large minors in graphs with given independence number.}
\newblock {Discrete Math.} 311 (2011), 2203--2215.

\bibitem{BLW}
{J. Balogh, J. Lenz and H. Wu,}
\newblock {\em Complete minors, independent sets, and chordal graphs.}
\newblock {Discuss. Math. Graph Theory} 31 (2011), 639--674.

\bibitem{bohman-keevash} T. Bohman and P. Keevash, \textit{The early evolution of the H-free process.} Inventiones Math. 181 (2010), 291--336.

\bibitem{BCE} B. Bollob\'as, P. Catlin and P. Erd\H{o}s,
{\em Hadwiger's conjecture is true for almost every graph.} {European J. Combin.} 1 (1980), 195--199.

\bibitem{D-M}
{P. Duchet and H. Meyniel,}
\newblock {\em On {H}adwiger's number and the stability number.}
\newblock In {Graph Theory}, vol.~62 of {North-Holland Mathematics
  Studies}. North-Holland, (1982), 71--73.

\bibitem{D-Y}
Z. Dvo\v r\'ak and L. Yepremyan, \emph{Independence number in triangle-free graphs
  avoiding a minor}, preprint, arXiv:1907.12999.

\bibitem{fox}
{J. Fox,}
\newblock {\em Complete minors and independence number.}
\newblock {SIAM J. Discrete Math.} 24 (2010), 1313--1321.

\bibitem{kawa1}
{K. Kawarabayashi, M. D. Plummer and B. Toft}
\newblock {\em Improvements of the theorem of {D}uchet and {M}eyniel on {H}adwiger's}
  conjecture.
\newblock {J. Combin. Theory Ser. B} 95 (2005), 152--167.

\bibitem{kawa2}
{K. Kawarabayashi and Z. X. Song,}
\newblock { \em Independence number and clique minors.}
\newblock {J. Graph Theory} 56 (2007), 219--226.

\bibitem{kost}
A.~V. Kostochka, \emph{The minimum {H}adwiger number for graphs with a given
  mean degree of vertices.} Metody Diskret. Analiz. no.~38 (1982), 37--58.

\bibitem{K-S-T}
T. K{\"o}v{\'a}ri, V. S{\'o}s, and P. Tur{\'a}n, \emph{On a
  problem of K. Zarankiewicz.} Colloq. Math. 3 (1954), 50--57.

\bibitem{kuratowski}
K. Kuratowski, Sur le probleme des courbes gauches en topologie.
{\em Fund. Math.} 16 (1930), 271--283.


\bibitem{B-M}
M. Krivelevich and B. Sudakov, \emph{Minors in expanding graphs.} Geom. Funct.
  Anal. 19 (2009), 294--331. 

\bibitem{K-O}
D. K\"{u}hn and D. Osthus, \emph{Complete minors in {$K_{s,s}$}-free graphs.}
  Combinatorica 25 (2005), 49--64.

\bibitem{lovasz-survey}
L. Lov\'{a}sz, \emph{Graph minor theory.} Bull. Amer. Math. Soc. 43 (2006), 75--86.

\bibitem{mader}
   W. Mader, {\em Homomorphies\"{a}tze f\"{u}r Graphen.} {Math. Ann. 178} (1968), 154--168.
   
\bibitem{maffraymeyniel}
{F. Maffray and H. Meyniel,}
\newblock {\em On a relationship between {H}adwiger and stability numbers.}
\newblock {Discrete Math.} 64 (1987), 39--42.

\bibitem{norin-survey}
S. Norin, \emph{New tools and results in graph minor structure theory.}
  Surveys in combinatorics 2015, London Math. Soc. Lecture Note Ser., vol. 424,
  Cambridge Univ. Press, Cambridge, 2015, 221--260.
  
\bibitem{norinpostlesong}
S. Norin, L. Postle and X. Song, Breaking the degeneracy barrier for coloring graphs with no $K_t$ minor, preprint, arXiv:1910.09378. 

\bibitem{Pedersentoft}
{A.S. Pedersen and B. Toft,}
\newblock {\em A basic elementary extension of the {D}uchet--{M}eyniel theorem.}
\newblock {Discrete Math.} 310 (2010), 480 -- 488.

\bibitem{PlummerStiToft}
{M. D. Plummer, M. Stiebitz and B. Toft,}
\newblock {\em On a special case of {H}adwiger's conjecture.}
\newblock {Discuss. Math. Graph Theory} 23 (2003), 333--363.

\bibitem{Postle1} 
L. Postle, {\em Further progress towards Hadwiger’s conjecture}, preprint, arXiv:2006.11798.

\bibitem{Postle2} 
L. Postle, {\em An even better Density Increment Theorem and its application to Hadwiger's Conjecture}, preprint, arXiv:2006.14945.

\bibitem{R-S}
N. Robertson and P.~D. Seymour, \emph{Graph minors. {XX}. {W}agner's
  conjecture.} J. Combin. Theory Ser. B 92 (2004), 325--357.

\bibitem{hadwiger-5}
{N. Robertson, P. Seymour and R. Thomas,}
\newblock {\em Hadwiger's conjecture for {$K_6$}-free graphs.}
\newblock {Combinatorica} 14 (1993), 279--361.

\bibitem{hadwiger-survey}
{P. Seymour,}
\newblock {\em Hadwiger's conjecture.}
\newblock {In {J}ohn {F}orbes {N}ash {J}r. and {M}ichael {T}h. Rassias,
  eds., Open Problems in Mathematics, Springer\/} (2015), 417--–437.

\bibitem{thom}
A. Thomason, \emph{An extremal function for contractions of graphs.} Math.
  Proc. Cambridge Philos. Soc. 95 (1984), 261--265. 

\bibitem{wagner-4-col}
{K. Wagner,}
\newblock {\em {\"U}ber eine {E}igenschaft der ebenen {K}omplexe.}
\newblock {Math. {A}nn.} 114 (1937), 570--590.

\bibitem{wagner-kuratowski}
K. Wagner, \"Uber eine Eigenschaft des Satzes von Kuratowski. {\em
Deutsche Mathematik} 2 (1937), 280--285.

\bibitem{woodall}
{D. R. Woodall,}
\newblock {\em Subcontraction-equivalence and {H}adwiger's conjecture.}
\newblock {J. Graph Theory} 11 (1987), 197--204.


\end{thebibliography}
\end{document}